\documentclass[12pt,amscd]{amsart}
\usepackage{tikz}
\usepackage[all]{xy}
\usepackage{graphicx}
\usepackage{amsmath,amsxtra,amssymb,latexsym, amscd,amsthm}
\usepackage{indentfirst}
\usepackage[mathscr]{eucal}
\usepackage[pagebackref=true]{hyperref}

\newtheorem{thm}{Theorem}[section]
\newtheorem{cor}[thm]{Corollary}
\newtheorem{lem}[thm]{Lemma}

\theoremstyle{definition}
\newtheorem{defn}[thm]{Definition}
\newtheorem{exm}[thm]{Example}

\newtheorem{rem}[thm]{Remark}

\DeclareMathOperator{\ZZ}{\mathbb {Z}}
\DeclareMathOperator{\RR}{\mathbb {R}}

\DeclareMathOperator{\reg}{reg}

\def\mat{\operatorname{mat}}

\def\R {\mathcal R}

\def\a {\mathbf a}
\def\b {\mathbf b}

\def\e {\mathbf e}

\def\k {\mathrm k}
\def\p {\mathbf p}

\begin{document}

\title {Regularity of normal Rees algebras of edge ideals of graphs}

\author{Cao Huy Linh}
\address{University of Education, Hue University, 34 Le Loi, Hue City, Vietnam}
\email{caohuylinh@hueuni.edu.vn}

\author{Quang Hoa Tran}
\address{University of Education, Hue University, 34 Le Loi, Hue City, Vietnam}
\email{tranquanghoa@hueuni.edu.vn}
\author{Thanh Vu}
\address{Institute of Mathematics, VAST, 18 Hoang Quoc Viet, Hanoi, Vietnam}
\email{vuqthanh@gmail.com}

\date{}

\dedicatory{Dedicated to the memory of Professor J\"urgen Herzog (1941-2024)}
\commby{}
\begin{abstract}
    We classify all graphs for which the Rees algebras of their edge ideals are normal and have regularity equal to their matching numbers.
\end{abstract}

\maketitle

\section{Introduction}\label{sect_intro}
In March 2023, we were fortunate to participate in the lectures on binomial ideals by Professor J\"urgen Herzog in the CoCoa school at Hue University. Motivated by his beautiful lectures and his recent joint work with Professor Takayuki Hibi on the regularity of Rees algebras of edge ideals we asked him about the regularity of Rees algebras of edge ideals of odd cycles. Professor Herzog said that it is a good research problem, so we started by trying to compute the regularity of Rees algebras of edge ideals of odd cycles. It turns out that a simple modification of the argument in \cite{HH} gives us the answer. Furthermore, it was known at that time that if $\R(G)$ is normal and $G$ is a K\"onig graph or has a perfect matching then $\reg (\R(G)) = \mat(G)$, where $\R(G)$ is the Rees algebra of the edge ideal of $G$, $\reg$ denotes the Castelnuovo-Mumford regularity, and $\mat(G)$ is the matching number of $G$. It is not hard to find a non K\"onig graph which does not have a perfect matching but still has $\reg (\R(G)) = \mat(G)$. But they are essentially gluing of K\"onig graphs to perfect matching graphs. We will prove a more precise statement below in this work which grew out of influential lectures and the work of Professor Herzog. We humbly dedicate this to our admiral teacher, Professor J\"urgen Herzog.

Let us now recall the notion of the Rees algebra of the edge ideal of a simple graph. Let $G$ be a simple graph on the vertex set $V(G)$ and edge set $E(G) \subseteq V(G) \times V(G)$. Assume that $V(G) = [n] = \{1,\ldots,n\}$. The Rees algebra of the edge ideal of $G$ over a field $\k$, denoted by $\R(G)$ is the subalgebra of $\k[x_1,\ldots, x_n,t]$ generated by $x_1,\ldots,x_n$ and $x_ix_jt$ where $\{i,j\}$ is an edge of $G$. In \cite{HH}, Herzog and Hibi proved that $\R(G)$ is normal if and only if $G$ satisfies the odd cycle condition \cite{OH} and has at most one non-bipartite connected component. Furthermore, in this case, $\mat(G) \le \reg (\R(G)) \le \mat(G) +1$. By \cite[Theorem 4.2]{CR}, \cite[Corollary 2.3]{HH}, and \cite[Corollary 3.2]{NN}, we have that $\reg(\R(G)) = \mat(G)$ when $G$ is a bipartite graph, a perfect matching graph, or a connected K\"onig graph. But it was not known whether $\reg(\R(G)) = \mat(G)$ or $\mat(G) + 1$ when $G$ is an odd cycle, which is the starting point of this work.

We will now introduce relevant graph concepts to state our main result. A subset $T \subseteq V(G)$ is an independent set of $G$ if $E(G) \cap T \times T = \emptyset$. For a subset $U$ of $V(G)$, we denote by $N_G(U)$ the set of neighbors of $U$ in $G$. 

\begin{defn}
    A graph $G$ is called a Tutte-Berge graph if there exists an independent set $T$ of $G$ such that 
    \begin{equation}\label{eq_TB}
        |T| = |N_G(T)| + |V(G)| - 2 \mat(G).
    \end{equation}
\end{defn}
The number $|V(G)|-2 \mat(G)$ is the number of uncovered vertices by a maximum matching, which was described by Tutte \cite{T} and Berge \cite{B}. When $G$ has a perfect matching, $T= \emptyset$ satisfies Eq. \eqref{eq_TB}. When $G$ is K\"onig, then a maximum independent set of $G$ satisfies Eq. \eqref{eq_TB}. Hence, perfect matching graphs and K\"onig graphs are two extremes of Tutte-Berge graphs. Tutte-Berge graphs are precisely the graphs for which $\reg (\R(G)) =\mat(G)$ when $\R(G)$ is normal.
\begin{thm}\label{thm_main}
    Let $G$ be a simple graph. Assume that $G$ has at least two edges and $\R(G)$ is normal. Then $\reg (\R(G)) = \mat(G)$ if and only if $G$ is a Tutte-Berge graph.
\end{thm}
The main ingredients for the proof of Theorem \ref{thm_main} are the Gallai-Edmonds Structure Theorem and the description of the edge polytope of a graph of Ohsugi and Hibi \cite{OH}. To state a characterization of Tutte-Berge graphs, we recall the Gallai-Edmonds decomposition. Let $G$ be a simple graph. Denote by $D(G)$ the set of all vertices in $G$ which are not covered by at least one maximum matching of $G$. Let $A(G)$ be the set of vertices in $V(G) \backslash D(G)$ adjacent to at least one vertex in $D(G)$ and $C(G) = (V(G) \backslash D(G) ) \backslash  A(G)$. The decomposition $V(G) = D(G) \cup A(G) \cup C(G)$ is called the Gallai-Edmonds decomposition.

\begin{thm}\label{thm_characterization} Let $G$ be a simple graph. Then $G$ is Tutte-Berge if and only if $D(G)$ consists of isolated vertices only.
\end{thm}

In the next section, we establish properties of Tutte-Berge graphs and prove Theorem \ref{thm_characterization}. In Section \ref{sec_reg}, we prove Theorem \ref{thm_main}.

\section{Tutte-Berge graphs}\label{sec_Tutte_Berge}
In this section, we classify all Tutte-Berge graphs. We first introduce relevant concepts. We refer to the beautiful exposition \cite{LP} for unexplained terminology and further information.

\begin{defn} Let $G$ be a simple graph. A matching in $G$ is a set of edges, no two of which share an endpoint. The matching number of $G$, denoted by $\mat(G)$, is the size of a maximum matching of $G$. A perfect matching is a matching that covers every vertex of the graph. A graph that has a perfect matching is also called a perfect matching graph.
\end{defn}
For a subset $U \subset V(G)$, we denote by $G\backslash U$ the induced subgraph of $G $ on $V(G) \setminus U.$ When $U=\{u\}$, we use $G\backslash u$ instead of $G\backslash \{u\}.$ We denote by $\alpha(G)$ the maximum size of an independent set of $G$.
\begin{defn} Let $G$ be a simple graph.
\begin{enumerate}
    \item $G$ is factor-critial if for every vertex $v$ of $G$, $G\backslash v$ has a perfect matching.
    \item $G$ is K\"onig if $\alpha(G) + \mat(G) = |V(G)|$.
\end{enumerate}
\end{defn}

Gallai \cite{G1, G2} and Edmonds \cite{E} independently proved the following structure theorem.

\begin{thm}[Gallai-Edmonds Structure Theorem]\label{thm_GE}Let $G$ be a simple graph, and $D(G), A(G)$ and $C(G)$ be defined as above. Then
\begin{enumerate}
    \item the components of the subgraph induced by $D(G)$ are factor-critical,
    \item the subgraph induced by $C(G)$ has a perfect matching,
    \item if $M$ is any maximum matching of $G$, it contains a near-perfect matching of each component of $D(G)$, a perfect matching of each component of $C(G)$ and matches all vertices of $A(G)$ with vertices in distinct components of $D(G)$,
    \item $2\mat(G) = |V(G)| - c(D(G)) + A(G)$, where $c(D(G))$ denotes the number of components of the graph spanned by $D(G)$.
\end{enumerate}
\end{thm}

We now have some preparation lemmas.

\begin{lem}\label{lem_critical} Let $G$ be a factor-critical graph. Assume that $|V(G)|>1$. Then for any independent set $T$ of $G$ we have $|T| \le |N_G(T)|.$ In particular, if $|V(G)| > 1$ then $G$ is not Tutte-Berge.
\end{lem}
\begin{proof}
    By definition, we have that $|V(G)|$ is odd. Let $T$ be an independent set of $G$. Since $|V(G)|>1$, there exists a vertex $v$ of $G$ such that $v \notin T$. Since $G\backslash v$ has a perfect matching and contains $T$, we deduce that
    $$|T| \le |N_{G\backslash v}(T)| \le | N_G(T)|.$$
    The conclusion follows.
\end{proof}

\begin{lem}\label{lem_mat_addition} Let $G$ be a simple graph and $U \subseteq V(G)$ a subset of vertices of $G$. We denote by $G_1$ and $G_2$ the induced subgraphs of $G$ on $U$ and $V(G)\backslash U$, respectively. Assume that there exists a maximum matching $M$ of $G$ and a partition $M = M_1 \cup M_2$ such that $M_1$ is contained in $G_1$ and $M_2$ is contained in $G_2$. Then $\mat(G) = \mat(G_1) + \mat(G_2)$.    
\end{lem}
\begin{proof}
    Since a matching of $G_1$ and a matching of $G_2$ give a matching of $G$, we deduce that $\mat(G) \ge \mat(G_1) + \mat(G_2)$. The existence of a maximum matching $M$ in the hypothesis implies that $\mat(G) \le \mat(G_1) + \mat(G_2)$. The conclusion follows.
\end{proof}

\begin{lem}\label{lem_bound_T} Let $G$ be a simple graph and $T \subseteq V(G)$ an independent set of $G$. Then 
$$|T| \le |N_G(T)| + |V(G)| - 2 \mat(G).$$    
\end{lem}
\begin{proof} Assume by contradiction that there exists a graph $G$ and an independent set $T$ of $G$ such that $|T| > |N_G(T)| + |V(G)| - 2 \mat(G)$. Let $G$ be such a graph of the smallest size. In particular, $G$ is connected. If $G$ has a perfect matching, then $|T| \le |N_G(T)|$ for any independent set $T$, so we must have $G$ does not have a perfect matching. If $A(G) = \emptyset$ then $G = D(G)$ is factor-critical, which is a contradiction to Lemma \ref{lem_critical}. Hence, $A(G) \neq \emptyset$. Let $v$ be any element of $A(G)$. Let $M$ be any maximum matching of $G$. By the Edmonds-Gallai Structure Theorem, $M$ contains an edge of the form $vw$ with $w$ in some component $D_2$ of $D(G)$. Let $G_2$ be the induced subgraph of $G$ on $V(D_2) \cup \{v\}$ and $G_1$ be the induced subgraph of $G$ on $V(G) \setminus V(G_2)$. By Theorem \ref{thm_GE} and Lemma \ref{lem_mat_addition}, we have that $G_2$ has a perfect matching and $\mat(G) = \mat(G_1) + \mat(G_2)$. Let $T_1 = T \cap V(G_1)$ and $T_2 = T \cap V(G_2)$. Then we have $N_G(T) \ge N_{G_1}(T_1) + N_{G_2}(T_2)$. Thus, we have 
$$|T_1| + |T_2| > |N_{G_1}(T_1) | + |N_{G_2}(T_2)| + |V(G_1)| - 2 \mat(G_1).$$
Since $G_2$ has a perfect matching, $|T_2| \le |N_{G_2}(T_2)|$. Therefore, we must have 
$$|T_1| > |N_{G_1}(T_1) | + |V(G_1)| - 2 \mat(G_1),$$
which is a contradiction, as $G_1$ is strictly smaller than $G$. The conclusion follows.
\end{proof}

\begin{lem}\label{lem_components} Let $G$ be a simple graph. Then $G$ is Tutte-Berge if and only if each connected component of $G$ is Tutte-Berge.   
\end{lem}
\begin{proof}
    The conclusion follows from the definition and Lemma \ref{lem_bound_T}.
\end{proof}

The following properties of Tutte-Berge graphs make it a natural class containing K\"onig graphs and perfect matching graphs.
\begin{lem}\label{lem_inheritance_of_TB}  Let $G$ be a Tutte-Berge graph and $U \subseteq V(G)$ a subset of vertices of $G$. We denote by $G_1$ and $G_2$ the induced subgraph of $G$ on $U$ and $V(G)\backslash U$, respectively. Assume that there exists a maximum matching $M$ of $G$ and a partition $M = M_1 \cup M_2$ such that $M_1$ is contained in $G_1$ and $M_2$ is contained in $G_2$. Then $G_1$ and $G_2$ are Tutte-Berge.    
\end{lem}
\begin{proof} By Lemma \ref{lem_mat_addition}, $\mat(G) = \mat(G_1) + \mat(G_2)$. Let $T$ be an independent set of $G$ such that $|T| = |N_G(T)| + |V(G)| - 2 \mat(G)$. We denote by $T_1 = T \cap U$ and $T_2 = T \cap (V(G)\backslash U)$. We have that $|N_G(T)| \ge N_{G_1}(T_1) + N_{G_2}(T_2)$. Hence,
$$|T_1| + |T_2| \ge \left( |N_{G_1}(T_1)| + |V(G_1)| - 2\mat(G_1) \right ) + \left( |N_{G_2}(T_2)| + |V(G_2)| - 2\mat(G_2) \right ).$$
By Lemma \ref{lem_bound_T}, we deduce that $|T_1| = |N_{G_1}(T_1)| + |V(G_1)| - 2\mat(G_1)$ and $|T_2| = |N_{G_2}(T_2)| + |V(G_2)| - 2\mat(G_2)$. The conclusion follows.    
\end{proof}

\begin{lem}\label{lem_perfect} Assume that $G$ is a Tutte-Berge graph and $D(G)$ has no isolated vertices. Then $D(G) = \emptyset$ and $G$ has a perfect matching.    
\end{lem}
\begin{proof} By Lemma \ref{lem_components}, we may assume that $G$ is connected. We prove by induction on $|V(G)|$. If $|A(G)| = 0$ then $D(G) = \emptyset$ or $G = D(G)$. By Lemma \ref{lem_critical}, we must have $D(G) = \emptyset$ and $G$ has a perfect matching. Assume by contradiction that $|A(G)| \ge 1$. Let $u$ be an element of $A(G)$. As in the proof of Lemma \ref{lem_bound_T}, let $M$ be any maximum matching of $G$. Then $M$ contains an edge that connects $u$ to a connected component $D_2$ of $D(G)$. Let $G_1 = G \backslash (D_2 \cup \{u\})$ and $G_2$ be the induced subgraph of $G$ on $D_2 \cup \{u\}$. By Theorem \ref{thm_GE} and Lemma \ref{lem_inheritance_of_TB}, $G_1$ is Tutte-Berge and $G_2$ has a perfect matching. We will prove that $D(G_1)$ has no isolated vertices. Let $w$ be any element of $D(G_1)$. Then there exists a maximum matching $M_1$ of $G_1$ such that $w \notin M_1$. $M_1$ and a maximum matching of $G_2$ form a maximum matching of $G$. Hence, $w \in D(G)$. In other words, $D(G_1) \subseteq D(G) \setminus D_2$ and $A(G_1) \subseteq A \setminus u$. Assume by contradiction that $w \in D(G_1)$ is an isolated vertex. Then, the connected component of $D(G)$ containing $w$ must have vertices in $A(G_1) \subseteq A(G) \backslash u$. In other words, $D(G) \cap A(G) \backslash u \neq \emptyset$. This is a contradiction. Thus, $D(G_1)$ has no isolated vertices. By induction, $G_1$ has a perfect matching. Hence, $G$ itself has a perfect matching. In other words, $D(G) = \emptyset$. The conclusion follows.   
\end{proof}

\begin{lem}\label{lem_isolatedpoints} Assume that $G$ is a Tutte-Berge graph. Then, $D(G)$ consists of isolated vertices only.    
\end{lem}
\begin{proof} By Lemma \ref{lem_components}, we may assume that $G$ is connected. Assume by contradiction that $D(G)$ has a connected component $G_1$ that is not an isolated vertex. By Theorem \ref{thm_GE}, $G_1$ is factor-critical. Let $G_2 = G \setminus V(G_1)$. Since $G_1 \subseteq D(G)$, there exists a maximum matching $M$ of $G$ that uncovers a vertex of $G_1$. This implies that $\mat(G) = \mat(G_1) + \mat(G_2)$. By Lemma \ref{lem_inheritance_of_TB}, we deduce that $G_1$ is Tutte-Berge, which is a contradiction to Lemma \ref{lem_critical}. The conclusion follows.
\end{proof}

We are now ready for the proof of a characterization of Tutte-Berge graphs.
\begin{proof}[Proof of Theorem \ref{thm_characterization}] By Lemma \ref{lem_isolatedpoints}, it remains to prove the sufficiency condition. Let $G_1$ be the induced subgraph of $G$ on $D(G) \cup A(G)$. By Theorem \ref{thm_GE}, $\mat(G_1) = |A(G)|$ and $\mat(G) = \mat(G_1) + |C(G)|/2$. We have that $D(G)$ is an independent set of $G$ and 
\begin{align*}
    |D(G)| &= |V(G_1)| - \mat(G_1) = |A(G)| + |V(G_1)| - 2 \mat(G_1)\\
    &= |N_G(D(G)) | + |V(G)| - 2 \mat(G).
\end{align*}
Hence, $G$ is Tutte-Berge.
\end{proof}

\begin{exm} The following graph is Tutte-Berge but is not K\"onig nor has a perfect matching.
   
\begin{center}
\begin{tikzpicture}[roundnode/.style={draw,shape=circle,fill=blue,minimum size=1mm}]
        \node[circle,fill,inner sep=1pt]      (u1)        at (-0.5,0.5)             {};
         \node[circle,fill,inner sep=1pt]      (u2)        at (-0.5,-0.5)             {};
          \node[circle,fill,inner sep=1pt]      (u3)        at (0,0)             {};
          \node[circle,fill,inner sep=1pt]      (u4)        at (1,0)             {};
          \node[circle,fill,inner sep=1pt]      (u5)        at (2,0)             {};
          \node[circle,fill,inner sep=1pt]      (u6)        at (2.5,0.5)             {};
          \node[circle,fill,inner sep=1pt]      (u7)        at (2.5,-0.5)             {};
            \draw (u1) -- (u3);
            \draw (u2) -- (u3);
            \draw (u3) -- (u4);
            \draw  (u4) -- (u5);
            \draw  (u4) -- (u6);
            \draw (u4) -- (u7); 
            \draw  (u5) -- (u6);
            \draw (u5) -- (u7); 
            \draw (u6) -- (u7); 
          \end{tikzpicture}
\end{center}   
\end{exm}

\begin{rem}
\begin{enumerate}
\item The Edmonds' blossom algorithm \cite{E} gives a polynomial time algorithm for the Gallai-Edmonds structure decomposition. Hence, it also yields a polynomial time algorithm for determining if a graph is Tutte-Berge.
    \item When $D(G)$ consists of isolated vertices only the induced subgraph of $G$ on $D(G) \cup A(G)$ is K\"onig. Hence, a Tutte-Berge graph decomposes into a K\"onig graph and a perfect matching graph.
    \item  Deming \cite{D} and Sterboul \cite{S} independently gave the first characterization for K\"onig graphs. Theorem \ref{thm_characterization} is a natural analog of the characterization of K\"onig graphs given by Lov\'asz \cite[Lemma 3.3]{L}.
\end{enumerate}
\end{rem}

\section{Regularity of normal Rees algebras of edge ideals}\label{sec_reg}

In this section, we compute the regularity of normal Rees algebras of edge ideals of graphs. First, we recall the description of the edge polytope of a graph by Ohsugi and Hibi \cite{OH}.

Let $G$ be a simple graph on $n$ vertices. We denote by $\e_1,\ldots, \e_n$ the canonical bases of $\RR^n$. The edge polytope of $G$, denoted by $P_G$ is the convex hull of $\{ \e_i + \e_j \mid \{i,j\} \text{ is an edge of } G\}$. Let $L$ be the hyperplane $L = \{x \in \RR^n \mid x_1 + \cdots + x_n = 2\}$. For each $i=1,\ldots, n$, we denote by $H_i^+$ the half-space $H_i^+ = \{ x \in \RR^n \mid x_i \ge 0\}$. For each independent set $T$ of $G$, we denote by $H^-_{G,T}$ the half-space 
$$H_{G,T}^- = \left \{x \in \RR^n \mid \sum_{i \in T} x_i \le \sum_{j \in N_G(T)} x_j \right \}.$$

Let $T$ be an independent set of $G$, the bipartite graph induced by $T$, denoted by $B_{G}(T)$, is the graph with vertex set $V(B_G(T)) = T \cup N_G(T)$ and edge set $E(B_G(T)) = \{ \{v,w\} \mid v \in T, w\in N_G(T)\}$. 

\begin{defn} A vertex $v$ of $G$ is said to be regular in $G$ if each connected component of $G \setminus v$ has at least one odd cycle.
\end{defn}

\begin{defn}
    An independent set $T$ of $G$ is said to be fundamental in $G$ if it satisfies the following conditions
    \begin{enumerate}
        \item the bipartite graph $B_G(T)$ induced by $T$ is connected;
        \item if $T \cup N_G(T) \neq V(G)$, then each connected component of $G \backslash V(B_G(T))$ has at least one odd cycle.
    \end{enumerate}
\end{defn}
We have the following description of $P_G$ \cite[Theorem 1.7]{OH}.
\begin{thm}\label{thm_OH} Assume that $G$ has at least one odd cycle. Let $R$ be the set of regular vertices of $G$ and $F$ the set of nonempty fundamental independent sets of $G$. Then 
$$P_G = L \cap \bigcap_{i \in R} H_i^+ \bigcap_{T \in F} H_{G,T}^-.$$
\end{thm}
Let $G^*$ be the cone graph over $G$, i.e., $V(G^*) = V(G) \cup \{n+1\}$ and $$E(G^*) = E(G) \cup \{ \{i,n+1\} \mid i =1, \ldots,n\}.$$

We now have some preparation lemmas.
\begin{lem}\label{lem_reg_G*} Assume that $G$ has more than one edge. Then each vertex $v \in V(G)$ is a regular vertex of $G^*$.    
\end{lem}
\begin{proof} Since $G^* \backslash v$ is the cone graph over $G\backslash v$, it is connected. Furthermore, since $G$ has more than one edge, $G\backslash v$ has at least one edge; this edge and the new vertex in $G^*$ form a triangle in $G^*\backslash v$. The conclusion follows.    
\end{proof}

\begin{lem}\label{lem_reg} Let $G$ be a simple graph on $n$ vertices. Assume that $G$ has at least two edges and $\R(G)$ is normal. Let $q_0 = \min \{ q \ge 1 \mid q (P_{G^*} \setminus \partial P_{G^*}) \cap \ZZ^{n+1} \neq \emptyset\}$, where $\partial P_{G^*}$ is the boundary of $P_{G^*}$. Then 
$$\reg (\R(G)) = n + 1 - q_0.$$
\end{lem}
\begin{proof} The conclusion follows from the proof of \cite[Theorem 2.2]{HH}.
\end{proof}

\begin{lem}\label{lem_red_1} Assume that $q < n$ and $\a = (a_1,\ldots,a_{n+1}) \in q (P_{G^*} \backslash \partial P_{G^*})$. If $a_i > 1$ then $\b \in q(P_{G^*} \backslash \partial P_{G^*})$ where $\b = \a + \e_{n+1} - \e_i$.
\end{lem}
\begin{proof}Since $a_i > 1$, we have that $b_i \ge 1$. Thus, $\b \in H_j^+ \backslash \partial H_j^+$ for all $j = 1, \ldots,n+1$. Now, let $T$ be an independent set of $G^*$. If $T = \{n+1\}$, then we have $b_{n+1} = 2q - (b_1 + \cdots + b_n)$. Since $q < n$ and $b_i > 0$ for all $i= 1, \ldots, n$, we deduce that $b_{n+1} < n \le \sum_{i=1}^n b_i$. Hence, $\b \in H_{G^*,T}^- \backslash \partial H_{G^*,T}^-.$ Now, assume that $T$ is an independent set of $G$. If $i \in T$ then 
\begin{align*}
    \sum_{j \in T} b_j  &= \sum_{j \in T} a_j -1 < \sum_{j\in N_G(T)} a_j + a_{n+1} - 1  =  \sum_{j\in N_G(T)} b_j + b_{n+1} -2.
\end{align*}
Now, assume that $i \notin T$. Then 
\begin{align*}
    \sum_{j \in T} b_j  &= \sum_{j \in T} a_j  < \sum_{j\in N_G(T)} a_j + a_{n+1} - 1 \le \sum_{j\in N_G(T)} b_j + b_{n+1} -1.
\end{align*}
Hence, $\b \in H_{G^*,T}^-$ for all independent sets $T$ of $G$. The conclusion follows.
\end{proof}

\begin{lem}\label{lem_inP} Let $G$ be a simple graph. Let $\p = (1,1,\ldots,1,n - 2\mat(G))$ be a point in $\RR^{n+1}$. Then $\p \in qP_{G^*}$, where $q = n - \mat(G)$.
\end{lem}
\begin{proof}
    Clearly, $\p \in H_i^+$ for all $i = 1, \ldots, n+1$. Let $T$ be an independent set of $G^*$. Then either $T = \{n+1\}$ or $T$ is an independent set of $G$. If $T = \{n+1\}$, then $n - 2 \mat(G) < n$, so $\p \in H_{G^*,T}^-$. If $T$ is an independent set of $G$, the conclusion follows from the definition and Lemma \ref{lem_bound_T}.
\end{proof}

\begin{lem}\label{lem_additivity_fundamental} Let $G$ be a simple graph with connected components $G_1, \ldots, G_c$. Assume that $T_1,\ldots,T_c$ are fundamental independent sets of $G_1,\ldots,G_c$. Then $T = T_1 \cup \cdots \cup T_c$ is a fundamental independent set of $G^*$.    
\end{lem}
\begin{proof}
    Since $B_{G^*}(T) = B_{G_1}(T_1) \cup \cdots \cup B_{G_c}(T_c) \cup \{n+1\}$, $B_{G^*}(T)$ is connected. Furthermore, $G^* \setminus B_{G^*}(T) = (G_1 \backslash B_{G_1}(T_1)) \cup \cdots \cup (G_c \backslash B_{G_c}(T_c))$. The conclusion follows from the definition of fundamental independent sets.
\end{proof}
    
\begin{lem}\label{lem_Tutte_Berge_witness} Assume that $G$ is a Tutte-Berge graph. Then there exists a fundamental independent set $T$ of $G^*$ such that $|T| = |N_G(T)| + |V(G) | - 2\mat(G) $.    
\end{lem}
\begin{proof} By Lemma \ref{lem_additivity_fundamental}, we may assume that $G$ is connected. If $G$ is bipartite, we can take $T$ to be the maximum independent set of $G$. Thus, we may assume that $G$ is not bipartite. If $G$ has a perfect matching, we may take $T = \emptyset$. Thus, we may assume that $G$ does not have a perfect matching. By Theorem \ref{thm_characterization}, we have that $|P| = |N_G(P)| + |V(G)| - 2 \mat(G)$, where $P = D(G)$. By definition $V(B_G(P)) = D(G) \cup A(G)$. Let $H_1, \ldots, H_c$ be the connected components of $C(G)$. Since $C(G)$ has a perfect matching, $H_1, \ldots, H_c$ have a perfect matching. For each $i = 1, \ldots, c$, we set 
$$T_i  = \begin{cases}
    \emptyset &\text{ if } H_i \text{ is non-bipartite},\\
    \text{ a maximum indepdent set of } H_i & \text{ if } H_i \text{ is bipartite}.
\end{cases}$$
Let $T = P \cup T_1 \cup \cdots \cup T_c$. Then $T$ is an independent set of $G$ and have $|T| = |N_G(T)| + |V(G)| - 2 \mat(G)$. Furthermore, $G \setminus B_G(T) = \bigcup_j H_j$, where the union is taken over the indices $j$ such that $H_j$ is non-bipartite. Note that $B_{G^*}(T) = B_G(T) \cup \{n+1\}$, hence $B_{G^*}(T)$ is connected. By definition, $T$ is a fundamental set of $G^*$. The conclusion follows.
\end{proof}
We are now ready for the proof of Theorem \ref{thm_main}.
\begin{proof}[Proof of Theorem \ref{thm_main}] We may assume that $G$ does not have a perfect matching. Note that $\R(G)$ is normal by assumption. 

First, assume that $G$ is not Tutte-Berge. In particular, $n - 2\mat(G) > 0$ and for any independent set $T$ of $G$ we have $|T| < |N_G(T)| + n - 2 \mat(G)$. Let $\p = (1,\ldots,1,n-2\mat(G))$ be a point in $\ZZ^{n+1}$. Hence, $\p \in q(P_{G^*} \backslash  \partial P_{G^*})$ where $q = n - \mat(G)$. By Lemma \ref{lem_reg}, we have 
$$\reg (\R(G)) = n+1 - q_0 \ge n+1 - q = n+1 - (n-\mat(G)) = \mat(G) + 1.$$
By \cite[Theorem 2.2]{HH}, $\reg (\R(G)) = \mat(G) + 1$.

Now, assume that $G$ is Tutte-Berge. By \cite[Theorem 4.2]{CR} and \cite[Corollary 2.3]{HH}, we may assume that $G$ is not bipartite and does not have a perfect matching. Since $\R(G)$ is normal, $G$ can have at most one non-bipartite connected component. Let $G_1$ be the unique non-bipartite connected component of $G$. By Lemma \ref{lem_components}, $G_1$ is Tutte-Berge. By Lemma \ref{lem_Tutte_Berge_witness}, there exists a fundamental set $T_1$ of $G_1^*$  such that $|T_1| = |N_{G_1}(T_1)|  + |V(G_1)| - 2\mat(G_1)$. Let $G_2, \ldots, G_c$ be bipartite connected components of $G$. For each $i = 2, \ldots, c$, let $T_i$ be a maximum independent set of $G_i$. Then $T = T_1 \cup \cdots \cup T_c$ is a fundamental set of $G^*$ such that $|T| = |N_G(T)| + |V(G)| - 2 \mat(G)$. Since $G$ does not have a perfect matching, $|T| \ge 1$. 

We will now prove that $\reg (\R(G)) = \mat(G)$. By Lemma \ref{lem_reg}, it suffices to prove that $q_0 \ge n - \mat(G) + 1$. Indeed, let $\a = (a_1, \ldots,a_{n+1})$ be a point in $q_0 (P_{G^*} \backslash \partial P_{G^*})$. By Lemma \ref{lem_reg_G*}, each $i \in [n]$ is a regular vertex of $G^*$. Hence, $a_i \ge 1$ for all $i = 1,\ldots, n$. Now, assume by contradiction that $q_0 \le n - \mat(G) < n$. By Lemma \ref{lem_red_1}, we may assume that $a_i = 1$ for all $i = 1, \ldots, n$. Then $a_{n+1} = 2 q_0 - (a_1 + \cdots + a_n) = 2q_0 - n \le n - 2 \mat(G).$ In other words, $\a \notin q_0 ( H_{G^*,T} \backslash \partial H_{G^*,T}^-)$. It is a contradiction. The conclusion follows.
\end{proof}

\begin{cor} Let $G = C_{2n+1}$ be an odd cycle of length $2n + 1 \ge 3$. Then 
$$\reg (\R(G)) = \mat(G) + 1 = n+1.$$    
\end{cor}
\begin{proof}
    Since $G$ is not Tutte-Berge, the conclusion follows from Theorem \ref{thm_main}.
\end{proof}
\begin{rem}
    Let $G = K_n$ be a complete graph on $n$ vertices. Assume that $n\ge 3$ is odd. Then $G$ is not Tutte-Berge. Hence, $\reg (\R(G)) = \mat(G)+1$. This is also a special case of \cite[Corollary 2.12]{BVV}.
\end{rem}
\begin{rem} The assumption that $\R(G)$ is normal is crucial in Theorem \ref{thm_main}. As pointed out by Herzog and Hibi \cite[Example 2.4]{HH}, there is a perfect matching graph $G$ such that $\reg (\R(G)) > \mat(G)$. The reason is that when $\R(G)$ is not normal, one cannot use Danilov-Stanley Theorem \cite[Theorem 6.3.5]{BH}; hence, Lemma \ref{lem_reg} is no longer valid.    
\end{rem}

\subsection*{Data availability}

Data sharing does not apply to this article as no datasets were generated or analyzed during the current study.

\subsection*{Conflict of interest}

The authors have no relevant financial interests to disclose.

\section*{Acknowledgments}
Cao Huy Linh is partially supported by the Vietnam National Program for the Development of Mathematics 2021-2030 under grant number B2023-CTT-03. Quang Hoa Tran is supported by the Vietnam National Foundation for Science and Technology Development (NAFOSTED) under grant number 101.04-2023.07.

\end{document}